\documentclass[12pt]{article}
\usepackage{fullpage}
\usepackage{amssymb,calc,enumerate,booktabs}
\usepackage{amsmath}
\usepackage{amsthm}
\usepackage{etoolbox}
\usepackage{algorithmic,rotating}
\usepackage[algoruled,linesnumbered,longend,nofillcomment,noline]{algorithm2e}
\usepackage{units}
\usepackage{pdflscape}
\usepackage[affil-sl]{authblk}
\usepackage{enumitem}
\usepackage{mathtools}
\usepackage{enumitem}
\setlist[enumerate]{itemsep=0mm}
\usepackage{titlesec}
    \titleformat{\section}{\large\bfseries}{\thesection}{1em}{}

\def\Z{Z_{\max}}
\def\y{Z_{\min}}
\newcommand{\ff}{\mathcal{F}}
\newtheorem{theorem}{Theorem}

\title{\sf The quadratic balanced optimization problem\footnote{This research work was supported by an NSERC  Discovery grant and an NSERC discovery accelerator supplement awarded to Abraham P Punnen. }}
\author[1]{\normalsize \sf Abraham P. Punnen\thanks{apunnen@sfu.ca}}
\author[1]{\sf Sara Taghipour\thanks{ sara\_taghipour@sfu.ca}}
\author[1,2]{\sf Daniel Karapetyan\thanks{daniel.karapetyan@gmail.com}}
\author[3]{\sf Bishnu Bhattacharyya\thanks{bishnu@gmail.com}}
\affil[1]{\normalsize Department of Mathematics, Simon Fraser University Surrey, Central
City, 250-13450 102nd AV, Surrey, British Columbia, V3T 0A3, Canada}
\affil[2]{\normalsize ASAP Research Group, School of Computer Science, University of Nottingham, UK}
\affil[3]{Google, Mountain View, CA, USA}

\newenvironment{mathprog}[1][]%
{\par\bigskip\noindent\ignorespaces\begin{tabular*}{\textwidth}{@{} p{0.35\textwidth} @{\hspace{0.5em}} l} #1}%
{\end{tabular*}\\\smallskip\par\noindent\ignorespacesafterend}

\newcommand{\progline}[2][]{\hfill #1 & \begin{math}\displaystyle #2 \end{math}\\}

\date{}

\begin{document}

\maketitle

\begin{abstract}
\noindent We introduce the quadratic balanced optimization problem (QBOP) which can be used to model equitable distribution of resources with pairwise interaction. QBOP is strongly NP-hard even if the family of feasible solutions has a very simple structure.
Several general purpose exact and heuristic algorithms are presented.  Results of extensive  computational experiments are reported using randomly generated quadratic knapsack problems as the test bed. These results illustrate the efficacy of our exact and heuristic algorithms. We also show that when the cost matrix is specially structured, QBOP can be solved as a sequence of linear balanced optimization problems. As a consequence, we have several polynomially solvable cases of QBOP. \\

\end{abstract}

\section{Introduction}

Let $E=\{1,2,\ldots, m\}$ be a finite set and $\ff$ be a
family of non-empty subsets of $E$. It is assumed that $\ff$ is represented in a compact form of size polynomial in $m$ without explicitly listing its elements.  For each $(i,j)\in E\times E$, a cost $c_{ij}$ is prescribed.
 The elements of $\ff$ are called
feasible solutions and the $m\times m$ matrix $C=(c_{ij})$ is called the {\it cost matrix}. Then the {\it quadratic balanced optimization problem} (QBOP) is to find  $S\in \ff$ such that
  $$\max\{c_{ij}~:~ (i,j)\in S\times S\}-\min\{c_{ij}~:~ (i,j)\in S\times S\}$$
is as small as possible.

 QBOP is closely related to the {\it balanced optimization problem} introduced  by Martello et al~\cite{Martello1984}. To emphasize the difference between the balanced optimization problem of~\cite{Martello1984} and QBOP, we call the former a {\it linear balanced optimization problem} (LBOP).  Special cases of LBOP were studied by many authors~\cite{ b3, b0, Galil1988, h1, b4, Martello1984, b6}. Optimization problems with objective functions similar to that of LBOP have been studied by Zeitln~\cite{b15} for resource allocation, by Gupta and Sen~\cite{b23}, Liao and Huang~\cite{b24}, and Tegez and Vlach~\cite{y1,y2,y3} for machine scheduling,  by Ahuja~\cite{ah} for linear programming, by Scutell\`{a}~\cite{b20} for network flows, and by Liang et al.~\cite{liang} for workload balancing.  A generalization of LBOP where elements of $E$ are categorized has been studied by Bere\v{z}n\'{y}  and Lacko~\cite{b1,b2} and Grin\`{e}ov\'{a}, Kravecov\'{a}, and  Kul\'{a}\`{e}~\cite{g1}. Punnen and Nair~\cite{b9} studied LBOP with an additional linear constraint. Punnen and Aneja~\cite{b8} and Turner et al.~\cite{tur} studied the lexicographic version of LBOP\@. To the best of our knowledge, QBOP has not been studied in literature so far.

Most of the applications of LBOP discussed in literature translate into applications of QBOP by interpreting $c_{ij}$ as the pairwise interaction weight of elements $i$ and $j$ in $E$. To illustrate this, let us consider the following variation of the travel agency example of Martello et al.~\cite{Martello1984}. A North American  travel agency is planning to prepare a European tour package. Its clients travel from New York to London by a chartered flight. The clients have the option to choose a maximum of two tourist locations from an available set $S$ of locations. If a client chooses locations $i$ and $j$, then $c_{ij}$ is the total tour time. There are $n$ potential locations and the company wishes to choose $k=|S|$ locations to be included in the package so that the duration of tours for any pair of locations is approximately the same. This way, one can avoid waiting time of clients in London, after their tour and the whole group can return by the same chartered flight. The objective of the tour company can be represented as Minimizing $\max\{c_{ij} \, : \; (i,j)\in S\times S\}-\min\{c_{ij} \, : \; (i,j)\in S\times S\}$ while satisfying appropriate constraints.

Other applications of the model include balanced portfolio selection for managing investment accounts where risk estimates on pairs of investment opportunities are to be considered because of hedging positions and participant selection for psychological experiments where it is important that all the people in the group equally know each other.

The objective function of QBOP can be viewed as range of a covariance matrix $C$ associated with a combinatorial optimization problem. In this case, BQOP attempts to minimize a dispersion measure. Minimization of various measures of dispersion such as variance, absolute deviation from the mean etc.\ has been studied in the context of combinatorial optimization problems~\cite{kat,pa}. However, none of these studies take into consideration information from the covariance matrix which measures impact of pairwise interaction. This interpretation leads to other potential applications of our model.

In this paper we study QBOP and propose several general purpose algorithms. The polynomial solvability of these algorithms are closely related to that of an associated feasibility problem. QBOP is observed to be NP-hard even if the family $\ff$ of feasible solutions has very simple structure. We also investigate QBOP when the cost matrix $C$ has a decomposable structure, i.e., $c_{ij}=a_i+b_j$ or $c_{ij}=a_ib_j$. In each of these special cases, we show that QBOP can be solved in polynomial time whenever the corresponding LBOP can be solved in polynomial time. As a consequence, we have $O(m^2\log n)$ and $O(n^{6})$ algorithms for QBOP when $\ff$ is chosen as spanning trees of a graph on $n$ nodes and $m$ edges or perfect matchings on a $K_{n,n}$, respectively. Our general purpose exact algorithms can be modified into heuristic algorithms. Some sufficient conditions are derived to speed up these algorithms and their effect is analyzed using extensive experimentation in the context of quadratic balanced knapsack problems. We also compared the heuristic solutions with  exact solutions  and the results establish the efficiency of our heuristic algorithms.

The paper is organized as follows. In Section 2 we discuss the complexity of the problem and introduce notations and
definitions. Section 3 deals with exact and heuristic algorithms. In Section 4 we present our polynomially solvable special cases. In Section 5 we discuss the special case of the quadratic balanced knapsack problem. Experimental results are presented in Section 6. In Section 7, a generalization QBOP where interaction between $k$-elements are considered instead of two elements as in the case of QBOP.  Concluding remarks are presented in Section 8.

\section{Complexity and notations}

Without loss of generality, we assume that $c_{ij} \geq 0$ for otherwise we can add a large constant to all $c_{ij}$ values to get an equivalent problem with non-negative cost values. It may be noted that when $c_{ij} \geq 0$ for all $(i,j)\in E\times E$ and $c_{ii}=0$ for $i\in E$, QBOP reduces to the {\it quadratic bottleneck problem} (QBP)~\cite{bur,pz}. QBP is NP-hard even if $\ff$ is the collection of all subsets of $E$ with cardinality no more than $k$ for a given $k$, which depends on $m$~\cite{pz}. In fact, for such a problem, computing an $\epsilon$-optimal solution is also NP-hard for any $\epsilon > 0$ even if $c_{ij}\in \{0,1\}$ and $c_{ii}=0$~\cite{pz}. As an immediate consequence, it can be verified that for the corresponding instance of QBOP, computing an $\epsilon$-optimal solution is NP-hard for any $\epsilon > 0$. In contrast, the corresponding LBOP is polynomially solvable. Thus, the complexity of QBOP and LBOP are very different and QBOP apparently is a more difficult problem.\\

For a given cost matrix $C$ and $S\in \ff ,$ we denote
\begin{align*}
 &\Z(C,S) = \max\{c_{ij} : (i,j)\in S\times S\},\\   &\y(C,S)=\min\{c_{ij} : (i,j)\in S\times S\}\; \mbox{ and }\\
 &Z(C,S)=\Z(C,S)-\y(C,S).
 \end{align*}
For a given family of feasible solutions, we use the notation QBOP($C$) to indicate that the cost matrix under consideration for QBOP is $C$. Thus, QBOP($C$) and QBOP($C^*$), where $C\neq C^*$,  are two instances of QBOP with the same family of feasible solutions but different cost matrices $C$ and $C^*$ respectively.

For any two real numbers $\alpha$ and $\beta$ such that  $\alpha\leq \beta$ and cost matrix $C$, let $F(C,\alpha,\beta)=\{S \in \ff : \y(C,S) \geq \alpha \mbox{ and } \Z(C,S)\leq \beta  \}$ and $E(C,\alpha, \beta)=\{(i,j) : c_{ij} < \alpha \mbox{ or } c_{ij} > \beta\}$. Then the {\it quadratic feasibility problem} can be stated as follows: ``Given two real numbers $\alpha$ and $\beta$, where $\alpha\leq \beta$, test if $F(C,\alpha,\beta) \neq \emptyset$ and produce an $S\in F(C,\alpha,\beta)$ whenever $F(C,\alpha,\beta) \neq \emptyset$.''

Any solution $S\in \ff$ can be represented by its incidence vector $x=(x_1,x_2,\ldots ,x_m)$,  where

\begin{align*}
x_i = \begin{cases}
1 &\mbox{ if } i\in S,\\
0 & \mbox{ otherwise.}
\end{cases}
\end{align*}

The solution represented by an incidence vector $x$ is denoted by $S(x)$. Let $\ff_\text{I}$ be the the set of incidence vectors of elements of $\ff$. Consider the cost matrix $C^{\prime}$ given by
\begin{equation*}
c^{\prime}_{ij} = \begin{cases}
1 &\mbox{ if $(i,j)\in E(C,\alpha,\beta)$}, \\
0 & \mbox{ otherwise.}
\end{cases}
\end{equation*}

\noindent Then the quadratic feasibility problem has a `yes' answer if and only if the optimal objective function value of the {\it quadratic combinatorial optimization problem} (QCOP)

\begin{mathprog}
\progline[Minimize]{\sum_{i=1}^m\sum_{j=1}^m c^{\prime}_{ij}x_ix_j}
\progline[subject to]{x \in \ff_\text{I}, x_{j}=0 \mbox{ or } 1, j=1,2,\ldots, n.}
\end{mathprog}
is zero and if $x^0$ is the corresponding optimal  solution then  $S(x^0)\in F(C,\alpha,\beta)$. Thus, the quadratic feasibility problem can be solved by solving the QCOP.

The quadratic feasibility problem can also be viewed as the feasibility version of the {\it linear combinatorial optimization problem with conflict pairs} (LCOP), where the associated set of conflict pairs is precisely $E(C,\alpha,\beta)$; i.e, the quadratic feasibility problem has a ``yes'' answer if and only if the set
\begin{equation}\{ x : x\in \ff_\text{I}, x_i+x_j \leq 1 \mbox{ for } (i,j)\in E(C,\alpha,\beta)\}\end{equation}
is non-empty. For details on the LCOP we refer to~\cite{dar, Zhang2009}. The quadratic feasibility problem discussed above is closely related to the quadratic feasibility problem studied by Punnen and Zhang~\cite{pz} in the context of quadratic bottleneck problems.

\section{Exact and heuristic algorithms for QBOP}

Let us now consider some general results which are used in the subsequent sections to design algorithms for QBOP\@. Let $\ff^*=\{S_1,S_2,\ldots ,S_r\}$ be a subset of $\ff$ satisfying the following properties:
\begin{itemize}[leftmargin=3em,labelindent=3em]
\item[(P1)] There exists an $S_i\in \ff^*$ which is an optimal solution to QBOP,
\item[(P2)] $\Z(C,S_1) < \Z(C,S_2) < \cdots < \Z(C,S_r)$.

\end{itemize}
For any index $k$, $1\leq k\leq r$, let $\pi(k)$ be the index  such that $\Z(C,S_{\pi(k)})-\y(C,S_{\pi(k)})=\min\{\Z(C,S_{i})-\y(C,S_{i}): 1\leq i \leq k\}$. Then, clearly, $S_{\pi(r)}$ is an optimal solution to QBOP\@. Let $\Omega$ be a real number such that $\Omega \geq \y(C,S_i)$ for any optimal solution $S_i$ for QBOP in $\ff^*$.

\begin{theorem}\label{tth1}For any $1\leq k \leq r$, if $\Z(C,S_{\pi(k)})-\y(C,S_{\pi(k)})+\Omega \leq \Z(C,S_k)$ then $S_{\pi(k)}$ is an optimal solution to QBOP.
\end{theorem}
\begin{proof}Suppose $S_{\pi(k)}$ is not an optimal solution to QBOP\@. Then there is an optimal solution $S_i$ to QBOP in $\ff^*$ such that $i > k$. Thus, $\Z(C,S_i) - \y(C,S_i) < \Z(C,S_{\pi(k)})-\y(C,S_{\pi(k)})$. Then
\begin{align*}
\Z(C,S_i) & < \y(C,S_i) + \Z(C,S_{\pi(k)})-\y(C,S_{\pi(k)})\\
&\leq \Omega + \Z(C,S_{\pi(k)})-\y(C,S_{\pi(k)})\\
&\leq \Z(C,S_k).
\end{align*}
 Thus, by (P2), $i \leq k$, a contradiction.
\end{proof}

\begin{theorem}\label{tth2} If $S_{\pi(k)}$ is not an optimal solution to QBOP then there exists an optimal solution $S_q\in \ff^*$ such that $q > k$ and $\y(C,S_q) > \Z(C,S_k)-\Z(C,S_{\pi(k)})+\y(C,S_{\pi(k)})$
\end{theorem}
\begin{proof}
Since $S_{\pi(k)}$ is not optimal, by property (P1) there exists an optimal solution $S_q$ in $F^*$ such that $q > k$. Suppose $\y(C,S_q) \leq \Z(C,S_k)-\Z(C,S_{\pi(k)})+\y(C,S_{\pi(k)})$. Then $\y(C,S_q) \leq \Z(C,S_q)-\Z(C,S_{\pi(k)})+\y(C,S_{\pi(k)})$. Thus, $\Z(C,S_q) - \y(C,S_q) \geq \Z(C,S_{\pi(k)})- \y(C,S_{\pi(k)})$ establishing that $S_{\pi(k)}$ is also an optimal solution to QBOP, a contradiction.
\end{proof}
Theorems~\ref{tth1} and \ref{tth2} assist us in improving the average performance of our algorithms. Corresponding results can be obtained by considering solutions that satisfy another set of properties. Suppose $\ff^0=\{S_1,S_2,\ldots ,S_h\}$ be a subset of $\ff$ satisfying the following properties.
\begin{itemize}[leftmargin=3em,labelindent=3em]
\item[(P3)] There exists an $S_i\in \ff^0$ which is an optimal solution to QBOP
\item[(P4)] $\y(C,S_1) > \y(C,S_2) > \cdots > \y(C,S_h)$.
\end{itemize}
Choose an index $\sigma(k)$ such that $\Z(C,S_{\sigma(k)})-\y(C,S_{\sigma(k)})=\min\{\Z(C,S_{i})-\y(C,S_{i}): 1\leq i \leq k\}$. Then $S_{\sigma(h)}$ is an optimal solution to QBOP\@. Let $\Delta$ be a real number such that $\Delta \leq \Z(C,S_i)$ for any optimal solution $S_i$ for QBOP in $\ff^0$.

\begin{theorem}\label{th3}For any $1\leq k \leq h$, if $\Delta-\Z(C,S_{\sigma(k)})+\y(C,S_{\sigma(k)})\geq  \y(C,S_k)$ then $S_{\sigma(k)}$ is an optimal solution to QBOP\@.
\end{theorem}

\begin{theorem}\label{th4}If $S_{\sigma(k)}$ is not an optimal solution to QBOP then there exists an optimal solution $S_d\in \ff^0$ such that $d > k$ and $\Z(C,S_d) < \Z(C,S_{\sigma(k)})-\y(C,S_{\sigma(k)})+\y(C,S_k)$
\end{theorem}

The proofs of Theorems \ref{th3} and \ref{th4} can be obtained along the same lines as that of Theorems~\ref{tth1} and \ref{tth2} and hence are omitted.

Conditions similar to Theorems \ref{tth1} to \ref{th4} have been used by many authors in the context of different optimization problems involving linear terms~\cite{b18, b13, lp, b17, b9}. The effect of such conditions are not tested in the context of quadratic type problems. One of the goals of our experimental analysis was to test the efficacy of Theorems \ref{tth1} to \ref{th4} in the development of practical algorithms.

\subsection{The double threshold algorithm}
The basic idea of this algorithm is similar to that used by Martello et al.~\cite{Martello1984} for solving LBOP\@. The difference, however, is that we are using the quadratic feasibility problem discussed in the last section instead of a simple feasibility test considered in~\cite{Martello1984} for LBOP. This is a significant deviation as it alters the problem complexity substantially. The validity proof of our algorithm follows along the same line as that of the double threshold algorithm for LBOP discussed in~\cite{Martello1984}. We also use the conditions provided in Theorems~\ref{tth1} and \ref{tth2} to enhance our search for an optimal solution.

Let $w_1<w_2<\cdots <w_p$ be an ascending arrangement of distinct elements of the cost matrix $C$ and $w_{p+1}=\infty$. These $w_i$ values are the candidates for $\Z(C,S)$ and $\y(C,S)$ for any feasible solution $S$. The algorithm performs a  bottom-up sequential search by maintaining a lower threshold $L$ and an upper threshold $U$, and tests if $F(C,L,U)\neq \emptyset$. If the answer is `yes', the lower threshold is increased and if the answer is `no', the upper threshold is increased.  The lower and upper threshold values are chosen amongst $\{w_1,w_2,\ldots, w_p\}$. At any stage of the algorithm, if a feasible solution is obtained with the QBOP objective function value as zero, the algorithm is terminated since we have an optimal solution. Let $L = w_{\ell}$ and $U=w_u$ for some $\ell$ and $u$, $\ell\leq u$. If $F(C,w_{\ell},w_u)=\emptyset$ then $U$ is increased to $w_{u +1}$. Otherwise, we choose an $S\in F(C,w_{\ell},w_u)$ and $L$ can be increased to $w_{v+1}$, where $w_v=\min\{c_{ij} : (i,j)\in S\times S\}$, and the best solution identified so far is updated, if necessary. Note that $w_v \geq w_{\ell}$.

We also try to exploit the conditions of Theorems~\ref{tth1} and \ref{tth2} for early detection of an optimal solution or rapid increase in the lower threshold (and hence, possibly the upper threshold). Let $t \leq p$  be the total number of times $L$ is updated and let $\ff^*=\{S_1,S_2,\ldots ,S_{t}\}$ be the set of solutions generated. The indexes are selected such that $S_i$ is generated before $S_{i+1}$. Then $\ff^*$ satisfies the properties (P1) and (P2). Thus, the sufficient condition of Theorem~\ref{tth1} can be used to detect optimality in any iteration, whenever the condition is satisfied. If it  is satisfied then the best solution identified so far is indeed optimal and the algorithm terminates. Otherwise, we try to increase the lower threshold $L$ (and hence possibly the upper threshold $U$) rapidly using the conditions of Theorem~\ref{tth2}. If the algorithm is not terminated using any of the conditions discussed above, then the search completes  when $U$ or $L$ becomes $\infty$ and the best solution produced during the search is selected as the output which is an optimal solution to QBOP. A formal description of the bottom-up double threshold algorithm (BDT algorithm) is given in Algorithm~\ref{bdt}.

\begin{algorithm}[htb]
\caption{The BDT Algorithm}
\label{bdt}
        Let $w_1 <w_2 < \cdots < w_p$ be an ascending arrangement of distinct values of $e_{ij}$ for $ (i,j) \in E\times E $\;
        $l\leftarrow 1$; $u\leftarrow 1$; $\mathit{sol} \leftarrow\emptyset$; $\mathit{obj} \leftarrow\infty$\;
        Compute the parameter $\Omega$\tcc*[r]{See Section 6.3 for choice of $\Omega$}
        \While{$l \leq p$ and $u\leq p$}
        {
        \eIf{$F(C,w_l,w_u) \neq \emptyset$}
        {
        Choose an $S\in F(C,w_l,w_u)$\;
        \lIf{$Z(C, S) < obj$}{$\mathit{obj} \leftarrow Z(C, S)$; $\mathit{sol} \leftarrow S$}
        \lIf{$\mathit{obj} = 0$ or $\mathit{obj} + \Omega \leq w_t$}{\Return sol}
        Choose smallest $k$ such that $w_k > \max\{Z_\text{min}(C, S), Z_\text{max}(C, S) - \mathit{obj}\}$\;
        $l \leftarrow k$\;
        \lIf{$ w_l > w_u$}{$u\leftarrow k$}
        }
        {$u\leftarrow u+1$}
        }
        \Return{$\mathit{obj}$ and $\mathit{sol}$}
\end{algorithm}

With the assumption that the dominating complexity of this algorithm in each iteration is the complexity of testing the condition if $F(C,w_l,w_u) \neq \emptyset$ or not. If this test can be performed in $O(\phi_1(m))$ time, then the BDT-algorithm terminates in $O(m\phi_1(m))$ time. For most problems of practical interest, testing if $F(C,w_l,w_u) \neq \emptyset$ or not is NP-hard. By performing this test using heuristic algorithms, we can get a heuristic algorithm to solve QBOP\@. The computational issues associated with this approach are discussed in detail in the section of experimental analysis of algorithms.

Just like the BDT algorithm, it is possible to obtain another double threshold algorithm using a top-down search. In this case, we start with the upper and lower threshold values at $w_p$ and systematically decrease the threshold values. The algorithm makes use of Theorems~\ref{th3} and \ref{th4} to improve the search process, along similar lines as in the BDT algorithm where the threshold values are systematically increased. The resulting algorithm is called top-down double threshold algorithm (TDT algorithm). The detailed description of various steps of this algorithm can be easily constructed in view of the BDT algorithm and, therefore, omitted.

\subsection{Iterative bottleneck algorithms}

Let us now discuss two additional algorithms for solving QBOP which solve a sequence of quadratic bottleneck problems. The worst case complexities of these algorithms , in general, are higher than that of the BDT-algorithm, but their average performance is expected to be better. A quadratic bottleneck problem of type 1 (QBP1) is defined as
\begin{mathprog}[QBP1:]
\progline[Minimize]{\Z(C,S)}
\progline[subject to]{S \in \ff.}
\end{mathprog}

%
%
%
%
%
%
We denote an instance of QBP1 with cost matrix $C$ as QBP1($C$). The problem QBP1 was investigated by Burkard~\cite{bur} and Punnen and Zhang~\cite{pz}.

To develop our iterative bottleneck algorithms, we consider a generalization of QBOP, where  a  restriction on the lower threshold is imposed on the feasible solutions.  Consider the problem

\begin{mathprog}[QBOP1($C,\alpha$):]
\progline[Minimize]{\Z(C,S)-\y(C,S)}
\progline[subject to]{S \in \ff,}
\progline{\y(C,S) \geq \alpha}
\end{mathprog}

When $\alpha = \min\{c_{ij} : (i,j) \in E\times E\}$  QBOP1($C,\alpha$) reduces to QBOP(C).
Let $C^{\prime}$ be an $m\times m$ matrix defined by
\begin{equation*}
c^{\prime}_{ij} = \begin{cases}
M &\mbox{ if $c_{ij} < \alpha$},\\
c_{ij} & \mbox{ otherwise,}
\end{cases}
\end{equation*}
where $M$ is a large number.
\begin{theorem}\label{the1}Let $S^0$ be an optimal solution to QBP1 with cost matrix $C^{\prime}$ and $q$ be the index such that $w_q = \y(C,S^0)$.
\vspace{-2 mm}
\begin{enumerate}[label=(\arabic*)]
\item If $\Z(C,S^0) = M$ then QBOP1($C,\alpha$) is infeasible.
\item If $\Z(C,S^0) < M$ and $\Z(C,S^0)=\y(C,S^0)$ then $S^0$ is an optimal solution to QBOP1($C,\alpha$).
\item If conditions ($1$) and ($2$) above are not satisfied, then either $S^0$ is an optimal solution to QBOP1($C,\alpha$) or  an optimal solution to QBOP1($C^{\prime},\gamma$) is also optimal to QBOP1($C,\alpha$), where $\gamma=w_{q+1}$.
\end{enumerate}
\end{theorem}
\begin{proof} The proof of (1) and (2) are straightforward. Let us now prove (3).
Let $\chi=\{ S \in \ff :  \alpha \leq \y(C,S) \leq \y(C,S^0)\}$.
By definition of $\chi$
\begin{equation}\label{eq1}
\y(C,S^0) \geq \y(C,S) \mbox{ for all } S\in \chi.
\end{equation}
Since condition ($1$) of the theorem is not satisfied, by optimality of $S^0$ to QBP1 with cost matrix $C^{\prime}$, we have
\begin{equation}\label{eq2}
\Z(C,S^0) \leq \Z(C,S) \mbox{ for all } S\in \chi.
\end{equation}
Multiply inequality (\ref{eq1}) by $-1$ and adding to inequality (\ref{eq2}) we have
$Z(C,S^0) \leq Z(C,S)$ for all $S\in \chi$. Thus, either $S^0$ is an optimal solution to QBOP1(C,$\alpha$) or there exists an optimal solution $S$ to QBOP1($C,\alpha$) satisfying $\y(C,S) > \y(C,S^0)$ and the result follows.
\end{proof}

In view of Theorem~\ref{the1}, we can solve QBOP as a sequence of QBP1 problems. In each iteration, the algorithm maintains  a lower threshold $\alpha$ and constructs a modified cost matrix $C^{\prime}$ which depends on the value of $\alpha$. Then, using an optimal solution to QBP1 with cost matrix $C^{\prime}$, the lower threshold is systematically increased until infeasibility with respect to the threshold values is reached or optimality of one of the solutions generated so far is identified using condition (2) of Theorem~\ref{the1}. Let $F^1=\{S_{\rho_1},S_{\rho_2}, \ldots ,S_{\rho_t}\}$ be the set of solutions generated for various QBP1 problems, where $S_{\rho_i}$ is generated before $S_{\rho_{i+1}}$, $1\leq i \leq t-1$. Then by choosing $F^*=F^1$, these solutions satisfy properties (P1) and (P2). Thus, Theorem~\ref{tth1} can be used to detect optimality early and Theorem \ref{tth2} may be used to increase the lower threshold rapidly. If the algorithm is not terminated using an optimality condition, it compares all the solutions generated by the QBP1 solver and outputs the overall best solution with respect to the QBOP objective function. The resulting algorithm is called the {\it type 1 iterative bottleneck algorithm} (IB1 algorithm) and its formal description  is given in Algorithm~\ref{ib1}.

\vspace{4 mm}
\begin{algorithm}[H]
\caption{The IB1 Algorithm}
\label{ib1}

        Let $w_1 < w_2 < \cdots <w_p$ be an ascending arrangement of all distinct values of $\{c_{ij} : (i,j) \in E \times E \}$\;
        $C^ \prime \leftarrow C$; $\mathit{obj} \leftarrow \infty$ ; $\mathit{sol} \leftarrow \emptyset$, $M \leftarrow 1+w_p$; $z_0=w_p$\;
        Compute the parameter $\Omega$\;
        \While{$z_0 \neq M$}
        {
        Solve QBP1($C^{\prime}$)\;
        \lIf{If QBP1($C^{\prime}$) is infeasible}{\Return{$\emptyset$ and $\infty$}}
        \lElse{let $S$ be the solution of QBP1($C^{\prime}$)}
        $z_0\leftarrow \max\{c^{\prime}_{ij} : (i,j) \in S\times S\}$\;
        \If{$z_0< M$}
        {
        \lIf{$Z(C, S) < \mathit{obj}$}{$sol \leftarrow S$;
        $\mathit{obj} \leftarrow Z(C, S)$}
        \lIf{obj $ = 0$ or $\mathit{obj} + \Omega \leq z_0$}{\Return $\mathit{obj}$ and $\mathit{sol}$}
        $L\leftarrow\max\{Z_\text{min}(C,S),Z_\text{max}(C, S) - \mathit{obj}\}$\;

        $c^\prime _{ij}  \leftarrow  \left\{ \begin{array}{ll}
   c _{ij} \quad & \text{if } c_{ij} > L, \\
   M \quad & \text{otherwise} \end{array} \right.$ for each $i, j \in E$\;
        }
        }
        \Return{$\mathit{obj}$ and $\mathit{sol}$}
        \end{algorithm}

\vspace{4 mm}

The IB1 algorithm solves at most $m$ problems of the type QBP1. Thus, if QBP1 can be solved in $O(\phi_2(m))$ time, then QBOP can be solved in $O(m\phi_2(m))$ time. By solving QBP1 using a heuristic, we get a heuristic version of the IB1 algorithm.

QBOP can also be solved as a sequence of quadratic bottleneck problems of the maxmin type, which we call a {\it quadratic bottleneck problem of type 2} (QBP2). Formally, QBP2 can be stated as follows:
\begin{mathprog}[QBP2:]
\progline[Maximize]{\y(C,S)}
\progline[subject to]{S\in \ff.}
\end{mathprog}

QBP2 can be reformulated as QBP1 or the algorithms for QBP1~\cite{pz} can be modified to solve QBP2 directly.

Now, for any real number $\beta$, consider the problem:
\begin{mathprog}[QBOP2($C,\beta$):]
\progline[Minimize]{\Z(C,S)-\y(C,S)}
\progline[subject to]{S \in \ff,}
\progline{\Z(C,S) \leq \beta}
\end{mathprog}
When $\beta = \max\{c_{ij} : (i,j) \in E\times E\}$, QBOP2($C,\beta$) reduces to QBOP(C).
Define the cost matrix $\tilde{C}$ defined by
\begin{equation*}
\tilde{c}_{ij} = \begin{cases}
-M &\mbox{ if $c_{ij} > \beta$ or $c_{ij} < \alpha$} \\
c_{ij} & \mbox{ otherwise.}
\end{cases}
\end{equation*}
where $M$ is a large number.
\begin{theorem}\label{the2}Let $S^0$ be an optimal solution to QBP2 with cost matrix $\tilde{C}$  and $r$ be the index such that $w_r = \Z(C,S^0)$.
\vspace{-1.5 mm}
\begin{enumerate}
\item If $\y(\tilde{C},S^0) = -M$ then QBOP2(C,$\beta$) is infeasible.
\item If $\y(\tilde{C},S^0) > -M$ and $\Z(C,S^0)=\y(C,S^0)$ then $S^0$ is an optimal solution to QBOP2(C,$\beta$).
\item If conditions ($1$) and ($2$) are not satisfied, then either $S^0$ is an optimal solution to QBOP2(C,$\beta$) or an optimal solution to QBOP2($\tilde{C},\gamma$)  is also optimal to QBOP2(C,$\beta$) where $\gamma=w_{r+1}$.
\end{enumerate}
\end{theorem}
The proof of this theorem can be constructed by appropriate modifications in the proof of Theorem~\ref{the1} and hence is omitted.

In view of Theorem~\ref{the2}, we can solve QBOP as a sequence of QBP2 problems. In each iteration, the algorithm maintains a  lower threshold $\alpha$  and an upper threshold $\beta$ and construct a modified cost matrix $\tilde{C}$ which depends on the value of $\alpha$ and $\beta$.  Using an optimal solution to QBP2 with cost matrix $\tilde{C}$, the upper threshold is systematically decreased  and the process is continued until infeasibility  with respect to the threshold values is reached or optimality of one of the solutions generated so far  is identified using condition (2) of Theorem~\ref{the2}. Let $F^2=\{S_{\eta_1},S_{\eta_2}, \ldots ,S_{\eta_t}\}$ be the set of solutions generated for various QBP2 problems. The indexes are selected such that $S_{\eta_i}$ is generated before $S_{\eta_{i+1}}$, $1\leq i \leq t-1$. Then by choosing $F^0=F^2$, these solutions satisfy properties (P3) and (P4). Thus, Theorem~\ref{th3} may be used to detect optimality early in some cases and Theorem \ref{th4} may be used to decrease the upper threshold rapidly. If the algorithm is not terminated using an optimality condition, it compares the solutions generated by the QBP2 solver and outputs the overall best solution with respect to the QBOP objective function. The resulting algorithm is called the {\it type 2 iterative bottleneck algorithm} (IB2-algorithm). A formal description of the IB2 algorithm is omitted as it can be obtained by appropriate modifications of the IB1 algorithm.

The IB2 algorithm solves at most $m$ problems of the type QBP2\@. Thus, if QBP2 can also be solved in $O(\phi_3(m))$ time then QBOP can be solved in $O(m\phi_3(m))$ time. By solving QBP2 using a heuristic, we get a heuristic version of the IB2 algorithm.

\subsection{The double bottleneck algorithm}

Note that Algorithm IB1 sequentially increases the lower threshold value while the algorithm IB2 sequentially decreases the upper threshold value. The two algorithms can be combined to generate another algorithm that alternately increases the lower threshold and decreases the upper threshold. The operations of increasing the lower threshold or decreasing the upper threshold are carried out by solving a QBP1 and QBP2, respectively. The resulting algorithm is called the {\it double bottleneck algorithm} (DB-Algorithm). The validity of the DB-algorithm follows from Theorems \ref{the1} and \ref{the2} and the validity of algorithms IB1 and IB2.  A formal description of the DB Algorithm is given in Algorithm \ref{aa2}.

\vspace{4 mm}
\begin{algorithm}[htb]
\caption{The DB Algorithm}
\label{aa2}

        Let $w_1 <w_2 <\cdots <w_p$ be an ascending arrangement of all distinct values of $\{c_{ij} : (i,j) \in E\times E \}$\;
        $C^ \prime \leftarrow C$; $obj\leftarrow \infty $ ; $sol \leftarrow \emptyset$, $M\leftarrow 1+w_p$; $\bar{z}=w_p$; $\tilde{z}=w_1$\;
        \While{$\bar{z} \neq M$ and $\tilde{z}\neq -M$ }
        {
        Solve QBP1($C^{\prime}$)\;
        \lIf{If QBP1($C^{\prime}$) is infeasible}{\Return{$\emptyset$ and $\infty$}}
        \lElse{let $S$ be the solution of QBP1($C^{\prime}$)}
        $\bar{z} \leftarrow \max\{c^{\prime}_{ij} : (i,j) \in S \times S\}$\;
        \If{$\bar{z} < M$}
        {
        \lIf{$Z(C, S) < \mathit{obj}$}{$\mathit{sol} \leftarrow S$;
        $\mathit{obj} \leftarrow Z(C, S)$}
        \lIf{obj $ = 0$}{\Return $\mathit{sol}$ and $\mathit{obj}$}
        $c^\prime _{ij}  \leftarrow  \left\{ \begin{array}{ll}
   c_{ij} \quad & \text{if $c_{ij} > Z_\text{min}(C, S)$,} \\
   M \quad & \text{otherwise} \end{array} \right.$ for each $i, j \in E$\;
        }
        \BlankLine
        Solve QBP2($C^{\prime}$); let $S$ be the resulting solution\;
        $\tilde{z}\leftarrow \min\{c^{\prime}_{ij} : (i,j) \in S\times S\}$\;
        \If{ $\tilde{z} > -M$}
        {
        \lIf{$Z(C, S) < \mathit{obj}$}{$\mathit{sol} \leftarrow S$;
        $\mathit{obj} \leftarrow Z(C, S)$}
        \lIf{$\mathit{obj} = 0$ }{\Return $\mathit{sol}$ and $\mathit{obj}$}
        $c^{\prime} _{ij}  \leftarrow  \left\{ \begin{array}{ll}
   c _{ij} \quad & \text{if $c_{ij} < Z_\text{max}(C, S)$} \\
   -M \quad & \text{otherwise} \end{array} \right.$ for each $i, j \in E$\;
        }
        }
        \Return{opt and sol}
        \end{algorithm}

\section{Polynomially solvable cases}

Let us now consider a special cases of QBOP that can be solved in polynomial time where the cost matrix is {\it decomposable}. i.e.  $c_{ij}=a_i+b_j$ for given $a_i \geq 0$, $b_j\geq 0$ and $i,j\in E$. We denote such an instance of a QBOP by QBOP($a+b$). Let $Z^+(C,S)$ denote the objective function of QBOP($a+b$). Then,
\begin{align}
\nonumber Z^+(C,S)&=\max\{c_{ij} : (i,j)\in S\times S\} - \min\{c_{ij} : (i,j)\in S\times S\}\\
\nonumber &=\max\{a_i+b_j : (i,j)\in S\times S\} - \min\{a_i+b_j : (i,j)\in S\times S\}\\
 &=\max\{a_i : i\in S\}+\max\{b_i : i\in S\}-\min\{a_i : i\in S\} - \min\{b_i : i\in S\}\\
\label{pol1}&=\max\{a_i : i\in S\}+ \max\{-a_i : i\in S\} +\max\{b_i : i\in S\}- \min\{b_i : i\in S\}
\end{align}

Let $w_i$ be some prescribed weight of $i\in E$ and $g: \ff \rightarrow \mathbb{R}$. Duin and Volgenant~\cite{Duin1991} showed

\noindent that combinatorial optimization problems of the type

\begin{mathprog}[COP($g$):]
\progline[Minimize]{\max\{w_i: i\in S\} + g(S)}
\progline[subject to]{S\in \ff.}
\end{mathprog}
can be solved in $O(m\zeta(m))$, where $\zeta(m)$ is the complexity of minimizing $g(S)$ over $\ff$\@. Note that
\[
Z^+(C,S) = \max\{a_i : i\in S\} + g(S),
\]
where $g(S) = \max\{-a_i : i\in S\} + g_1(S)$ and $g_1(S)= \max\{b_i : i\in S\} - \min\{b_i : i\in S\}$. But minimizing $g_1(S)$ over $\ff$ is precisely the LBOP~\cite{Martello1984}. Thus, recursively applying the results of Duin and Volgenant~\cite{Duin1991}, QBOP($a+b$)  can be solved in $O(m^2\eta(m))$ time, where $\eta(m)$ is the complexity of an LBOP with the same family of feasible solutions as that of the QBOP($a+b$).

Another interesting polynomially solvable case is obtained when $c_{ij} = a_ib_j$ for $(i,j)\in E\times E$ where $a_i \geq 0, b_i \geq 0$ for $i=1,2,\ldots ,n$. The corresponding instance of QBOP is denoted by QBOP($ab$). Now, for any feasible solution $S\in \ff$,
\begin{align}
\nonumber Z^*(C,S)&=\max\{c_{ij} : (i,j)\in S\times S\} - \min\{c_{ij} : (i,j)\in S\times S\}\\
\nonumber &=\max\{a_ib_j : (i,j)\in S\times S\} - \min\{a_ib_j : (i,j)\in S\times S\}\\
 &=\max\{a_i : i\in S\}\max\{b_i : i\in S\}-\min\{a_i : i\in S\}\min\{b_i : i\in S\}.
\end{align}

Let $\alpha$ and $\beta$ be two real numbers such that $\alpha \leq \beta$ and $Q_{\alpha,\beta}=\{i : \alpha \leq a_i \leq \beta\}$. Also let $\gamma = \min\{b_i : i\in Q_{\alpha,\beta}\}$, $\delta = \max\{b_i : i\in Q_{\alpha,\beta}\}$ and $A=\{a_1,a_2,\ldots ,a_n\}$. Consider the constrained g-deviation problem:
\begin{mathprog}[GDP($\alpha,\beta$):]
\progline[Minimize]{\max\{\beta b_i : i\in S\}+ \max\{-\alpha b_i: i\in S\}}
\progline[subject to]{S\in \ff,}
\progline{\max\{b_i : i\in S\} \leq \delta,}
\progline{\min\{b_i: i\in S\} \geq \gamma,}
\progline{\max\{a_i : i\in S\} = \beta,}
\progline{\min\{a_i: i\in S\} = \alpha.}
\end{mathprog}
Let $S^{\alpha\beta}$ be an optimal solution to GDP($\alpha\beta$) with optimal objective function value $Z^*(C,S^{\alpha\beta})$. Choose $S^0$ such that
$$Z^*(C,S^0)=\min\{Z^*(C,S^{\alpha\beta}): (\alpha,\beta)\in A\times A, \alpha \leq \beta\}$$
\begin{theorem}\label{gdp} $S^0$ is an optimal solution to QBOP($ab$).\end{theorem}
\begin{proof}Let $F^{\alpha\beta}$ be the family of feasible solutions of GDP($\alpha\beta$). It is possible that $F^{\alpha\beta}=\emptyset$ and in this case we choose $S^{\alpha\beta}=\emptyset$ with objective function value a very large number.   Note that $Z^*(C,S)=\max\{\beta b_i : i\in S\}+ \max\{-\alpha b_i: i\in S\}$ for all $S\in F^{\alpha\beta}$. Thus, $Z^*(C,S^{\alpha\beta}) \leq Z^*(C,S)$ for all $S\in F^{\alpha\beta}$. Since $\ff = \cup\{F^{\alpha\beta} : (\alpha,\beta)\in A\times A, \alpha\leq \beta\}$ the result follows.\end{proof}
Thus by Theorem~\ref{gdp}, QBOP($ab$) can be solved by solving $O(n^2)$ problems of the type GDP($\alpha,\beta$) considering all $(\alpha,\beta)\in A\times A$ such that $\alpha\leq \beta$. But GDP($\alpha,\beta$) is a $g$-deviation problem~\cite{Duin1991} which can be solved as a sequence of bottleneck problems of the type

\begin{mathprog}[BP($\alpha,\beta$):]
\progline[Minimize]{ \max\{-\alpha b_i: i\in S\}}
\progline[subject to]{S\in \ff,}
\progline{\max\{b_i : i\in S\} \leq \delta,}
\progline{\min\{b_i: i\in S\} \geq \gamma,}
\progline{\max\{a_i : i\in S\} = \beta,}
\progline{\min\{a_i: i\in S\} = \alpha.}
\end{mathprog}
Now the bottleneck problem BP($\alpha,\beta$) can be solved by solving $O(\log n)$ feasibility problems using the binary search version of the well known threshold algorithm for bottleneck problems~\cite{ed}. But we need to take care the constraints in BP($\alpha,\beta$) associated with the feasibility routine embedded within the threshold algorithm and this can be achieved by solving a minsum problem of the type

\begin{mathprog}[MSP:]
\progline[Minimize]{ \sum_{i\in S}w_i}
\progline[subject to]{S\in \ff .}
\end{mathprog}
The value of $w_i$ in MSP depends on $\alpha, \beta$ and the threshold value used for the feasibility test. We omit the details of the selection of $w_i$ values which can easily be constructed by an interested reader.  Thus if MSP can be solved in $O(\eta(m))$ time, then QBOP($ab$) can be  solved in $O(n^2\log n\eta(m))$ time.

\section{The quadratic balanced knapsack problem}
Let us now consider a specific case of QBOP called the {\it quadratic balanced knapsack problem} (QBalKP) which can be defined as follows:
\begin{align}
\nonumber\mbox{Minimize } &\max\{c_{ij} : (i,j)\in E\times E, x_i=x_j=1\}- \min\{c_{ij} : (i,j)\in E\times E, x_i=x_j=1\}\\
\label{k1}\mbox{subject to }& \sum_{j=1}^ma_jx_j \geq b\\
\label{k2}&x_j \in \{0,1\} \mbox{ for } j \in E.
\end{align}
By choosing $\ff \subseteq 2^E$ such that $S\in \ff$ implies $\sum_{i\in S}a_i \geq b$ we can see that QBalKP is an instance QBOP\@. The compact representation of $\ff$ is given by the constraints (\ref{k1}) and (\ref{k2}). QBalKP can be used to model the travel agency example and portfolio selection examples discussed in section 1.   Since the quadratic bottleneck knapsack problem (QBotKP)~\cite{zp} is a special case of QBalKP and QBotKP is strongly NP-hard, QBalKP is also strongly NP-hard. (Note that the LBOP version of QBalKP is solvable in polynomial time.) QBalKP can be formulated as a mixed integer program:
\begin{mathprog}{}
\progline[Minimize]{u - v}
\progline[Subject to]{\sum_{j=1}^ma_jx_j \geq b,}
\progline{u \ge c_{ij}y_{ij} \mbox{ for } (i,j)\in E\times E, c_{ij}\neq 0,}
\progline{v \le c_{ij}y_{ij} + M(1 - y_{ij}) \mbox{ for } (i,j)\in E\times E, c_{ij}\neq M,}
\progline{y_{ij}-x_i \leq 0 \mbox{ for } (i,j)\in E\times E,}
\progline{y_{ij}-x_j \leq 0\mbox{ for } (i,j)\in E\times E,}
\progline{x_i+x_j-y_{ij}\leq 1 \mbox{ for } (i,j)\in E\times E,}
\progline{x_j \in \{0,1\}, \mbox{ for } j \in E,}
\progline{0 \le u, v \le M, }
\end{mathprog}
where $M=\max\{c_{ij}: 1\leq i,j\leq n\}$. Solving the mixed integer programming formulation given above becomes difficult as the problem size increases. However, we can use the general purpose algorithms developed in the previous section to solve QBalKP\@. To use the algorithms IB1, IB2, and DB, we can make use of the algorithm of Zhang and Punnen~\cite{zp} as the QBP1 (QBP2) solver. To apply the BDT and TDT algorithms, we need an algorithm to solve the corresponding quadratic feasibility problem. 

Recall that the {\it quadratic feasibility problem} is: ``Given two real numbers $\alpha$ and $\beta$, where $\alpha\leq \beta$, test if $F(C,\alpha,\beta) \neq \emptyset$ and produce an $S\in F(C,\alpha,\beta)$ whenever $F(C,\alpha,\beta) \neq \emptyset$.''

In section 2, we indicated that a quadratic feasibility problem can be solved as a combinatorial optimization problem with conflict pair constraints~\cite{dar, Zhang2009}. We now observe that the quadratic feasibility problem for QBKP can be solved by solving the maximum weight independent set problem (MWIP) on a graph with node set $E$ and edge set $E(\alpha,\beta)$. An integer programming representation of this MWIP is given below.

\begin{mathprog}[MWIP:]
\progline[Maximize]{\sum_{i=1}^ma_jx_j}
\progline[subject to]{x_i+x_j \leq 1 \mbox{ for } (i,j) \in E(\alpha, \beta),}
\progline{x_j \in \{0,1\}, \mbox{ for } j \in E.}
\end{mathprog}

Let $x^*=(x_1^*,x_2^*,\ldots ,x_m^*)$ be an optimal solution to the MWIP and $z^*$ be its optimal objective function value. Then $F(C,\alpha,\beta) \neq \emptyset$ if and only if $z^* \geq b$. Thus, we can use an MWIP solver to implement the algorithms discussed in the previous section  for the special case of QBalKP\@. The solution of the quadratic feasibility problem discussed above is closely related to the quadratic feasibility problem studied by Zhang and Punnen~\cite{zp} for the quadratic bottleneck knapsack problem with appropriate differences to handle the QBalKP objective.

\section{Computational Experiments}

In this section we report results of extensive experimental analysis conducted on randomly generated QBalKP instances.  The objective of the experiments is to assess the relative performance of various algorithms developed in section 3. We have implemented exact and heuristic versions of these algorithms and compared the outcomes in terms of  solution quality and computational time. Our experiments also examined the effectiveness of the conditions provided by Theorems \ref{tth1}, \ref{tth2}, \ref{th3}, and \ref{th4} for early detection of optimality and rapid advancement through the search intervals. All the experiments were conducted on an Intel~i7-2600 CPU based PC\@.
The algorithms are implemented in C\#, and CPLEX~12.4 was used to solve the mixed integer programming problems within our implementations.
x86-64 instruction set was used and concurrency was not allowed in our the algorithms as well as in CPLEX.

All the algorithms discussed in this paper (except the MIP formulation of QBalKP) require a feasibility test procedure and the dominating complexity of these algorithms in each iteration is that of this procedure. Recall that a feasibility test answers the question if there exists a feasible solution to the QBalKP and, whenever the answer is `yes', it generates such a solution. In section 5 we observed that this can be achieved by solving a maximum weight independent set problem. We can also use other variations of this approach to test feasibility and the empirical behavior of different variations could be different. Since the feasibility test is carried out several times in the algorithm, the effect of different variations of the feasibility tests could affect the the running time as well as solution quality (for heuristic algorithms). For definiteness and simplicity, we have restricted our
 implementation to three different feasibility test procedures which are summarized below:
\begin{enumerate}
	\item[FT1:] Solve the maximum weight independent set problem as described in section~5 and then compare the objective value $\sum_{i = 1}^m a_i x_i$ with $b$.
	If $\sum_{i = 1}^m a_i x_i \ge b$, the resulting solution $S(x)$ is a feasible solution to the QBalKP\@.
	Otherwise the answer is `no'. We used CPLEX~12.4 to solve the resulting integer program.
	\item[FT2:] Consider the maximum weight independent set problem as described in section~5.
	Choose the objective function coefficients to be zero\footnote{One can use any objective function in this feasibility test.  If the objective function is not a constant, it is useful to force the solver to stop after it finds the first feasible solution.  However, our experiments showed that using the original objective function in this feasibility test slows down the algorithms.} and add the new constraint
\[
\sum_{i = 1}^m a_i x_i \ge b
\]
to the formulation. We used CPLEX~12.4 to solve the resulting constrained maximum weight independent set problem.
The CPLEX solver stops as soon as a feasible solution is found.
When using this feasibility test procedure, we set the `MIP Emphasis' parameter of CPLEX to `Emphasize feasibility over optimality' which in our experiments provided the best performance.
	\item[FT3:] Solve the integer program defined in FT2 by providing a time limit for the mixed integer programming solver.  Note that if the solver fails to find a feasible solution in the allowed time limit, there is no guarantee that a feasible solution does not exist and, thus, using such a procedure in any of the algorithms turns an exact algorithm into a heuristic. Because of this heuristic decision, the properties (P1) and (P2) may not hold precisely. Nevertheless, we make a heuristic assumption that these properties hold and proceed accordingly.
\end{enumerate}
As indicated earlier, even for these special cases of feasibility tests, the solutions returned by these feasibility tests could be different and, thus, not only the execution time of each feasibility test but also the optimization process itself for a QBalKP algorithm may vary even for the exact feasibility tests FT1 and FT2.

We use the following notations to represent our algorithms under different parameter settings.
MIP stands for the mixed integer programming formulation of the QBalKP solved with CPLEX (the parameters of CPLEX are default)\@.
BDT, IB and DB denote the BDT, IB1 and DB algorithms, respectively, where the effect of Theorem~\ref{tth1} is suppressed.
By default, we use feasibility test FT2.
If `$^\text{B}$' is added to the name of an algorithm (such as BDT$^\text{B}$ or IB$^\text{B}$), feasibility test FT1 is used.
BDT$^\Omega$ and IB$^\Omega$ stand for the variations of BDT and IB, where  early optimality detection is guaranteed by Theorem~\ref{tth1} is enabled.
BDT$^t$ denotes the heuristic version of the BDT algorithm, where $t$ is the time limit prescribed for each feasibility test (FT3).  IB$^t$ and DB$^t$ denote the heuristic version of the corresponding algorithms, where $t$ is the time limit prescribed for each feasibility test (FT3) within the QBP1 solving procedure.

\subsection{Test problems}

Since this is the first time when the QBalKP is considered in the literature, we have developed a class of test instances for the problem.
These test instances are random problems constructed as follows.
For each test instance, we are given a triple $(m, \sigma, s)$, where $m > 1$ is an integer, $\sigma > 0$ and $0 \le s \le 1$.
We first generate an $m \times m$ matrix $C^{\prime} = (c^{\prime}_{ij})$, where $c^{\prime}_{ij}$ is a normally distributed random integer with mean $\mu = 0$ and standard deviation $\sigma$ as given.
Then the matrix $C=(c_{ij})$ is generated, where $c_{ij} = c^{\prime}_{ij} - \min_{rs} c^{\prime}_{rs}$. This guarantees that $c_{ij} \geq 0$.
Then, an $m$-vector $(a_i)$, where $a_i$ a uniformly distributed random integer in the range $0 \le a_i \le 1000$, is generated.
Finally, $b$ is selected as a uniformly distributed random integer in the range $\lfloor 250 m s \rfloor \le b \le \lfloor 750 m s \rfloor$.
Observe that $E[\sum_{i \in S} a_i] = E[b]$ if $S \subseteq E$ such that $|S| = ms$, where $E[x]$ is the expected value of $x$.
In other words, by varying the value of $s$, one can control the number of non-zeros in an optimal solution to the QBalKP instance.

It may be noted that the instances we generated are symmetric in the following sense; replacing $c_{ij}$ with $(\max_{i'j'} c_{i'j'}) - c_{ij}$ would not, on average, change the properties of the matrix $C$.
Hence, the algorithms IB1 and IB2 are expected to show similar average performance.  Thus, hereafter, we do not discuss the algorithm IB2 and denote IB1 as IB. Likewise, BDT and TDT algorithms are expected to have similar average performance. Thus, we do not consider the TDT algorithm in our experimental analysis and focus on the BDT algorithm.

Figure~\ref{fig:depend-s} indicates relative performance of the BDT and IB algorithms as a function of the parameter $s$.  We set $m = 100$ and $\sigma = 100$ in this experiment.

\begin{figure}
\begin{center}
\includegraphics[page=1,trim = 23mm 113mm 20mm 92mm, clip=true, width=.9\textwidth]{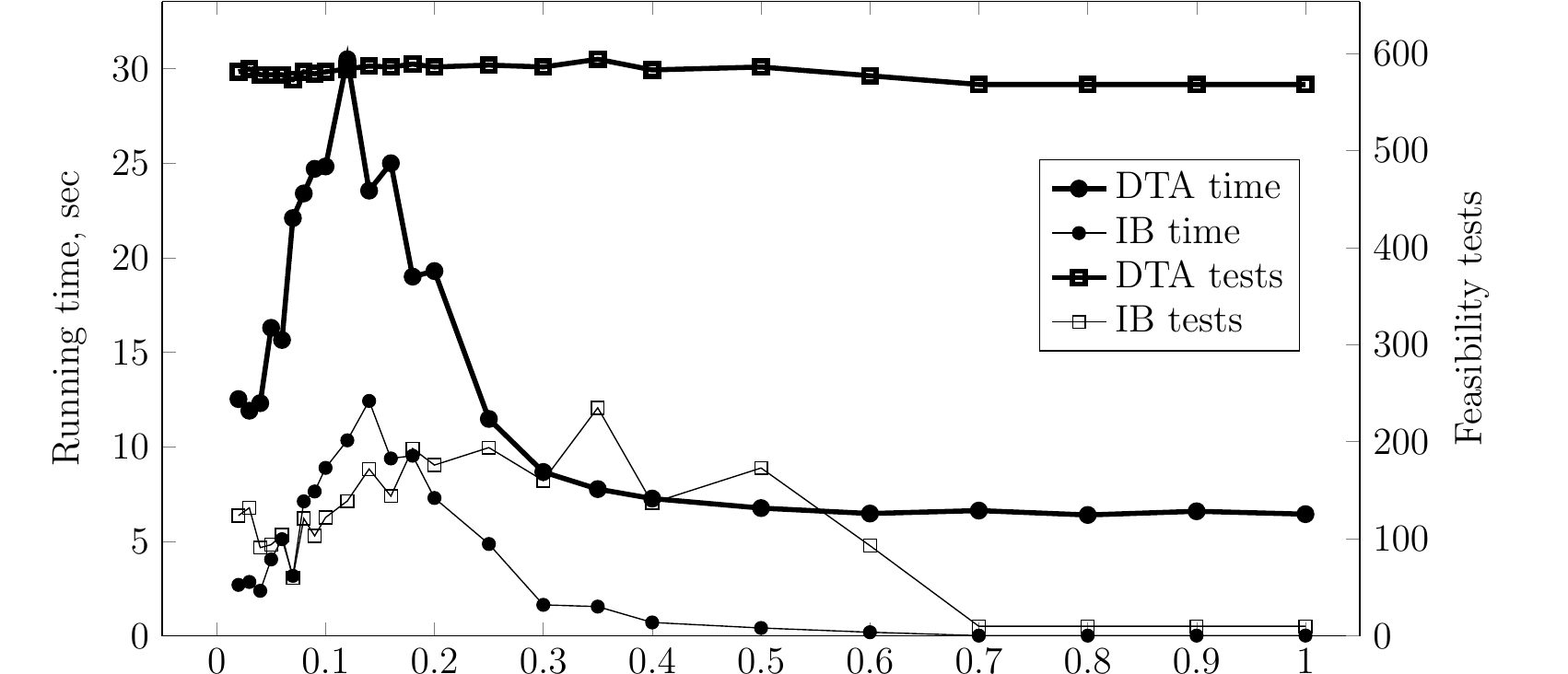}
\end{center}
\caption{Analysis of the random instances with different values of parameter $s$.  The lines with circle marks indicate the running time of the algorithms and the lines with square marks indicate the number of feasibility tests applied within the algorithms.}
\label{fig:depend-s}
\end{figure}

It appears that the instance with either large or small values of $s$ are relatively easy to solve.  For small values of $s$, each feasibility test takes only a small amount of  time since it is easy to find a feasible solution if $b$ is small.  For large values of $s$, each feasibility test also takes a relatively small time since many  such problems become  infeasible.  Another interesting observation is that the number of iterations of the BDT algorithm almost does not depend on $s$ (indeed, even if the problem is infeasible, the BDT algorithm will make $p$ iterations) while it varies significantly for the IB algorithm making it significantly faster for certain class of instances.

Figure~\ref{fig:depend-sigma} indicates the performance of the BDT and IB algorithms as a function of the parameter $\sigma$ of the instance.
The value of $s$ is set in this experiment to $s = 0.1$ and $m = 100$.

\begin{figure}
\begin{center}
\includegraphics[page=2,trim = 23mm 113mm 20mm 92mm, clip=true, width=.9\textwidth]{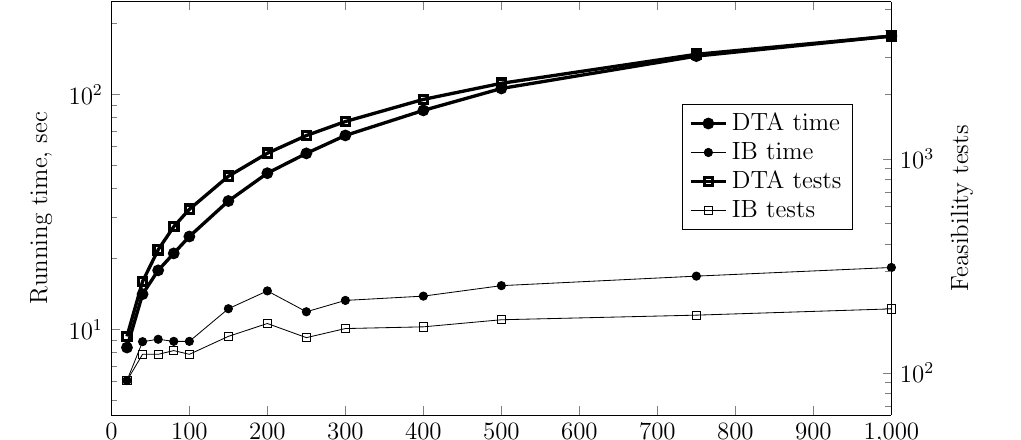}
\end{center}
\caption{Analysis of the random instances with different values of parameter $\sigma$.  The lines with circle marks indicate the running time of the algorithms and the lines with square marks indicate the number of feasibility tests applied within the algorithms.}
\label{fig:depend-sigma}
\end{figure}

One can see that the random instances become harder with  increase in the value of $\sigma$.
Indeed, a larger $\sigma$ leads to a larger number $p$ of distinct weights $c_{ij}$ which, in turn, increases the number of iterations of the algorithm.
However, $p$ is limited by $m^2$ and, thus, it grows slower than $\sigma$.
In fact, the number of iterations of the BDT algorithm is approximately proportional to $p$.
In contrast, the IB algorithm efficiently handles instances with large $p$.
Recall that it needs only $O(\log p)$ feasibility tests to solve each QBotKP subproblem~\cite{zp}.
Hence, for instances with large $p$, the IB algorithm appears preferable.
Also it follows from our experiments that the time needed for the feasibility test is almost independent on the value of $\sigma$.

Based on these preliminary observations, we set $s = 0.1$ and $\sigma = 100$ for the rest of the experiments to assure that the test problems generated are reasonably hard.

\subsection{Comparison of BDT, IB, DB and MIP}

Let us first evaluate performance of the basic algorithms proposed in this paper.
In Table~\ref{tab:basic}, we report the results of our experiments with the BDT$^\text{B}$, BDT, IB$^\text{B}$, IB, DB$^\text{B}$, DB and MIP algorithms.
The notations used in various columns of the table are explained below:
\begin{itemize}
	\item $m$ is the size of the test instance;
	\item `obj' is the objective value of the optimal solution to the problem;
	\item $p$ is the number of distinct $c_{ij}$ values in matrix $C$;
	\item $\Delta = \max_{i,j} c_{ij} - \min_{i,j} c_{ij}$;
	\item `Running time, sec' columns report the running time of each of the algorithms;
	\item `Feasibility tests' columns report the number of times feasibility tests are carried out within each of the algorithms. Note that in IB, IB$^\text{B}$, DB, and DB$^\text{B}$ algorithms, we are not explicitly solving feasibility problems. However, feasibility problems of similar nature are solved within the QBotKP solver that is used within these algorithms. Thus, the number of feasibility tests include the feasibility tests carried out within the QBotKP solver used, which is a variation of the algorithm by Zhang and Punnen~\cite{zp}.
\end{itemize}

The last row of the table reports the average values for corresponding columns.
However, it may be noted that the average running time can not be used to judge the algorithm's performance in general because the running times vary significantly from instance to instance.

The best result (running time and number of tests) for each instance is underlined.


Feasibility test FT2 provides a better performance than feasibility test FT1 in each case.
Indeed, according to Table~\ref{tab:basic}, the optimality of a solution to the maximum weight independent set problem does not reduce the number of feasibility tests, while reaching the optimal solution clearly takes more time than finding a feasible solution.

The IB algorithm clearly outperforms all other algorithms for each test instance with regards to both the number of feasibility tests and the running time.  The MIP algorithm turns out to be very slow for any practical instances.

\begin{sidewaystable}[ht] \centering
\setlength{\tabcolsep}{0.4em}
\footnotesize
\begin{tabular}{@{} l r r r @{} c @{} r r r r r r r @{} c @{} r r r r r r @{}}
\toprule
&&&&&\multicolumn{7}{c}{Running time, sec}&&\multicolumn{6}{c}{Feasibility tests}\\
\cmidrule(){6-12}
\cmidrule(){14-19}
$m$&obj&$p$&$\Delta$&\hspace*{1.5em}&BDT$^\text{B}$&BDT&IB$^\text{B}$&IB&DB$^\text{B}$&DB&MIP&\hspace*{1.5em}&BDT$^\text{B}$&BDT&IB$^\text{B}$&IB&DB$^\text{B}$&DB\\
\midrule
50&64&453&689&&2.7&2.2&0.6&\underline{0.5}&0.6&0.5&1.7&&463&462&94&\underline{87}&97&95\\
100&239&568&728&&131.5&22.0&56.2&\underline{8.0}&99.3&12.2&7018.1&&582&581&128&\underline{125}&235&209\\
150&204&622&772&&476.0&115.8&240.1&\underline{37.5}&282.7&43.1&13187.2&&642&643&191&\underline{188}&227&228\\
200&181&672&778&&938.1&217.7&475.3&\underline{38.8}&756.0&75.1&---&&684&684&117&\underline{101}&207&174\\
250&206&701&807&&3543.5&980.2&3806.5&\underline{473.5}&8449.6&916.4&---&&719&720&182&\underline{178}&347&348\\
\midrule
Avg.&179&603&755&&1018.4&267.6&915.7&\underline{111.7}&1917.6&209.5&---&&618&618&142&\underline{136}&223&211\\
\bottomrule
\end{tabular}
\caption{Comparison of basic algorithms.}
\label{tab:basic}
\vspace{2\baselineskip}
\begin{tabular}{@{} l r r @{} c @{} r r r r @{} c @{} c c @{} c @{} r r r r @{} c @{} r r r r @{}}
\toprule
&&&&\multicolumn{4}{c}{Running time, sec}&&\multicolumn{2}{c}{Early}&&\multicolumn{4}{c}{Iterations}&&\multicolumn{4}{c}{Feasibility tests}\\
\cmidrule(){5-8}
\cmidrule(){10-11}
\cmidrule(){13-16}
\cmidrule(){18-21}
$m$&obj&$\Omega$&\hspace*{1.5em}&BDT&BDT$^\Omega$&IB&IB$^\Omega$&\hspace*{1.5em}&BDT$^\Omega$&IB$^\Omega$&\hspace*{1.5em}&BDT&BDT$^\Omega$&IB&IB$^\Omega$&\hspace*{1.5em}&BDT&BDT$^\Omega$&IB&IB$^\Omega$\\
\midrule
50&64&466&&2.2&2.2&0.5&\underline{0.4}&&\checkmark&\checkmark&&462&432&10&\underline{9}&&462&440&87&\underline{86}\\
100&239&295&&22.0&19.2&\underline{8.0}&8.2&&\checkmark&\checkmark&&581&467&14&\underline{13}&&581&476&125&\underline{124}\\
150&204&346&&115.8&55.3&37.5&\underline{34.5}&&\checkmark&\checkmark&&643&477&21&\underline{20}&&643&485&188&\underline{186}\\
200&181&362&&217.7&120.1&38.8&\underline{37.1}&&\checkmark&\checkmark&&684&497&12&\underline{11}&&684&505&101&\underline{99}\\
250&206&348&&980.2&804.9&\underline{473.5}&503.0&&\checkmark&\checkmark&&720&554&20&\underline{19}&&720&563&178&\underline{177}\\
\midrule
Avg.&179&363&&267.6&200.3&\underline{111.7}&116.6&&&&&618&485&15&\underline{14}&&618&494&136&\underline{134}\\
\bottomrule
\end{tabular}
\caption{Analysis of efficiency of the early detection speed up.}
\label{tab:omega}
\end{sidewaystable}

\subsection{Early Detection of Optimality}

In this section we report the results of our experiments that assess the effectiveness of Theorem~\ref{tth1} for algorithms BDT and IB\@.
In Table~\ref{tab:omega}, we compare the results of the corresponding experiments.
The `$\Omega$' column reports the objective value of the optimal solution to the QBP2 version of the bottleneck knapsack problem (QBotKP2):
\begin{align}
\nonumber\mbox{Maximize } &\min\{c_{ij} : (i,j)\in E\times E, x_i=x_j=1\}- \min\{c_{ij} : (i,j)\in E\times E, x_i=x_j=1\}\\
\mbox{Subject to }& \sum_{j=1}^ma_jx_j \geq b\\
&x_j \in \{0,1\}, \mbox{ for } j=1,2,\ldots ,m.
\end{align}

This is used as the value of the parameter $\Omega$ in BDT$^\Omega$ and IB$^\Omega$ algorithms.
The `Early' columns report if early detection of optimality happened for the particular
 algorithm and test instance.
`Iterations' columns report the number of iterations of the algorithms.

In our experiments, early optimality detection happened for every test instance and every algorithm, where such an option was enabled.
For the BDT algorithm, early detection significantly reduced the number of iterations and improved the overall performance.
However, it reduces the number of iterations of the IB algorithm only by one in each run.
Since calculating of $\Omega$ is approximately equivalent to one iteration of the IB algorithm, the overall running time of the procedure almost did not change after enabling early detection.
\clearpage
Another interesting observation that we can make from Table~\ref{tab:omega} is that the number of iterations of the IB algorithm almost does not depend on the size of the problem.
Since each such iteration needs $O(\log p) = O(\log m)$ feasibility tests, the real running time of the IB algorithm per iteration is expected to be much higher than that of the BDT algorithm given large instances.

Recall that the value of $\Omega$ does not need to correspond to an optimal solution of the QBotKP2 but may be an upper bound to the optimal objective function value of QBotKP2\@.  Observe that the algorithm used in our experiments to solve the QBotKP2 maintains a search interval which depends on the values of two indices $\ell, u$ where $\ell \leq u$. When $u-\ell < 1$
the algorithm terminates and guarantees an exact optimal solution~\cite{zp}. This termination condition can be relaxed to $u-\ell < d$ to obtain a heuristic bound.
In our implementation, we used the termination condition with $u - l < d$, where $1 \le d \le p$.  Then the resulting upper bound $\Omega(d)$ can be calculated as $\Omega(d) = w_{u - 1}$.

\begin{figure}
\begin{center}
\includegraphics[page=3,trim = 23mm 118mm 20mm 100mm, clip=true, width=.9\textwidth]{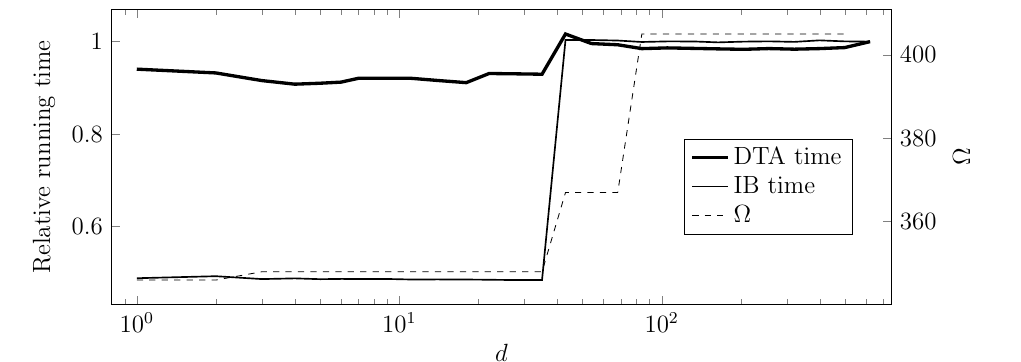}
\end{center}
\caption{Analysis of the early optimality detection for different values of $d$.  Observe that BDT$^{\Omega(1)} = \text{BDT}^{\Omega}$, and for $d = p$ we replace BDT$^{\Omega(p)}$ and IB$^{\Omega(p)}$ with the BDT and IB algorithms, respectively.  The relative running time is calculated as $t_d / t_1$, where $t_d$ is the running time of the BDT$^{\Omega(d)}$ or IB$^{\Omega(d)}$ algorithm and $t_1$ is the running time of the BDT or IB algorithm, respectively.  The experiments were conducted for a random instance of size $m = 150$.}
\label{fig:omega}
\end{figure}

The results of our experiments with the approximate value of $\Omega$ are reported in Figure~\ref{fig:omega}.  When the value of $d$ is relatively small, the upper bound $\Omega(d)$ is very close to the exact value of $\Omega$.  For $d > 40$, the upper bound becomes less accurate and the early detection happens in neither DTA$^{\Omega(d)}$ nor IB$^{\Omega(d)}$ algorithm.  Thus, the effect of using an upper bound $\Omega(d)$ instead of the optimal objective value $\Omega$ is relatively small.

\subsection{Comparison of Heuristics}

%

Recall that, by setting a time limitation to CPLEX when running a feasibility test, one can speed up the BDT, IB and DB algorithms at the cost of loosing optimality guarantee.
In Figure~\ref{fig:heuristics}, we show how the solution quality depends on the running time for each of the BDT$^t$, IB$^t$ and DB$^t$ algorithms.

\begin{figure}
\begin{center}
\includegraphics[page=4,trim = 23mm 114mm 20mm 100mm, clip=true, width=.9\textwidth]{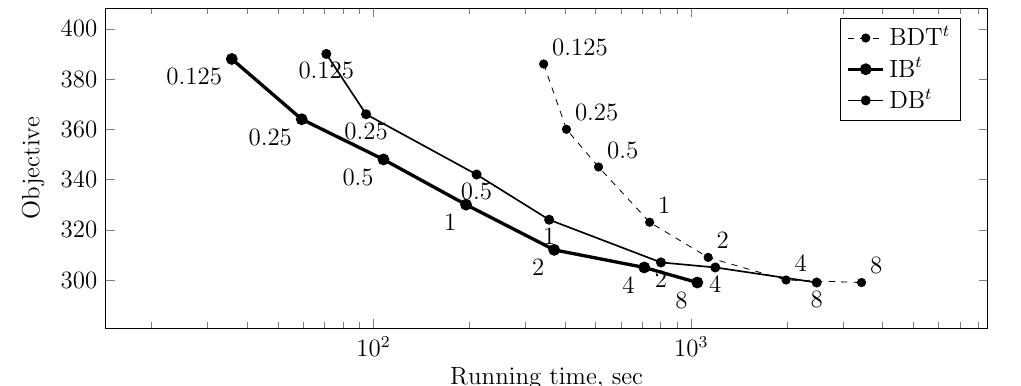}
\end{center}
\caption{Performance of the BDT$^t$, IB$^t$ and DB$^t$ heuristics.  The experiment is conducted for a random instance of size $m = 500$.  The values of $t$, in seconds, are reported near each node.}
\label{fig:heuristics}
\end{figure}

In our experiments, the running time and the solution quality of each of the heuristics monotonically depend on the parameter $t$.
Thus, the balance between the solution quality and the running time in the proposed heuristics can be efficiently controlled by $t$.
Note that the IB$^t$ algorithm, likewise its exact version, shows the best performance among the proposed heuristics.

We also compared the BDT$^t$, IB$^t$ and DB$^t$ heuristics with the MIP algorithm given a time limitation (in this case, we set the `MIP Emphasis' parameter of CPLEX to `Emphasize feasibility over optimality').
However, in our experiments, such a MIP heuristic showed very poor performance.

\section{Multinomial balanced optimization}

We now discuss a generalization of the quadratic balanced optimization. For any fixed integer $k$, a cost $c_{i_1i_2\ldots i_k}$ for $k$-tuple $(i_1,i_2,\ldots ,i_k)$ is prescribed. Consider the family $\ff$ of feasible solutions as in the case of QBOP. Then the {\it multinomial balanced optimization problem} (MBOP)is to find  an $S\in \ff$ such that
  $$\max\{c_{i_1i_2\ldots i_k}~:~ (i_1,i_2,\ldots ,i_k)\in S^k\}-\min\{c_{i_1i_2\ldots i_k}~:~ (i_1,i_2,\ldots ,i_k)\in S^k\}$$ is minimized,
where $S^k= \underbrace{S\times S\times \cdots \times S}_{k-\mbox{times}}$.

In QBOP, $c_{ij}$ is viewed as pairwise interaction cost while in MBOP we have `interaction cost' for $k$-tuples (i.e. cost for $k$ element ordered subsets of the ground set $E$.) The algorithms discussed in this paper extends in a natural way to the case of MBOP with an appropriate definition of the corresponding feasibility problem. We leave it for an interested reader to verify this claim. The number of distinct cost elements to be considered is $m^k$ as opposed to $m^2$ for the the case of QBOP. Thus, as $k$ increases the resulting algorithm could slow down significantly.

\section{Conclusion}
We introduced the combinatorial optimization model QBOP which can be used to model equitable distribution problems with pairwise interactions. The problem is strongly NP-hard even if the family of feasible solutions has a simple structure such as the collection of all subsets of a finite set with an upper bound on the cardinality of these subsets. Several exact and heuristic algorithms are provided along with detailed experimental analysis in the case of quadratic knapsack problems. Special cases of the problem with decomposable type cost matrices are discussed. It is shown that the complexity of the resulting QBOP depends on that of the corresponding LBOP.

It is not difficult to extend our results to the maximization version of the problem. By exploiting special structure of the family of feasible solutions $\ff$ and the structure of $C$, one may be able to obtain improved algorithms. These are interesting topics for further investigations, especially when real life situations warrant the study of such problems.\\

\noindent{\bf Acknowledgement:} We are thankful to the referees for their helpful comments which improved the presentation of the paper.

\end{document}